\documentclass{article}
\usepackage{amssymb,amsmath,amsthm,graphicx}

\def\qed{\hfill {\hbox{${\vcenter{\vbox{               %HOLLOW SQUARE
   \hrule height 0.4pt\hbox{\vrule width 0.4pt height 6pt
   \kern5pt\vrule width 0.4pt}\hrule height 0.4pt}}}$}}}

\def\utr{\, \underline{\triangleright}\, }
\def\otr{\, \overline{\triangleright}\, }

\newtheorem{theorem}{Theorem}

\newtheorem{proposition}[theorem]{Proposition}
\newtheorem{corollary}[theorem]{Corollary}

\newtheorem{observation}{Observation}
\theoremstyle{definition}
\newtheorem{example}{Example}
\newtheorem{definition}{Definition}
\newtheorem{remark}{Remark}

\date{}

\title{\Large \textbf{Quantum Enhancements and Biquandle Brackets}}

\author{Sam Nelson\footnote{Email: Sam.Nelson@cmc.edu. Partially supported by Simons Foundation collaboration grant 316709}\and
Michael E. Orrison\footnote{Email: orrison@hmc.edu.} \and
Veronica Rivera\footnote{Email: vrivera@g.hmc.edu.}} 

\begin{document}
\maketitle

\begin{abstract}
We introduce a new class of quantum enhancements we call \textit{biquandle 
brackets}, which are customized skein invariants for biquandle colored links.
Quantum enhancements of biquandle counting invariants form a class of knot and 
link invariants that
%arise from quantum 
%invariants of biquandle-colored links, i.e., functors from the category of
%biquandle-colored tangles to an $R$-module category for a commutative ring 
%$R$. 
%This class of invariants 
includes biquandle cocycle invariants and skein invariants such as the 
HOMFLY-PT polynomial as special cases, providing an explicit unification
of these apparently unrelated types of invariants. 
We provide examples demonstrating that 
the new invariants are not determined by the biquandle counting invariant,
the knot quandle, the knot group or the traditional skein invariants.
\end{abstract}

\parbox{5.5in} {\textsc{Keywords:} biquandles, biquandle brackets, 
quantum invariants, quantum enhancements of counting invariants

\smallskip

}

\section{\large\textbf{Introduction}}\label{I}

Biquandles, algebraic structures with axioms derived from the Reidemeister moves
for oriented knots, were introduced in \cite{FRS} and have been used to define
invariants of oriented knots and links in \cite{FJK,KR} and more. In 
particular, the number of biquandle colorings of an
oriented knot or link diagram $K$ by a finite biquandle $X$ 
defines a nonnegative integer-valued invariant
known as the \textit{biquandle counting invariant}, denoted 
$\Phi_X^{\mathbb{Z}}(K)$. An \textit{enhancement} of $\Phi_X^{\mathbb{Z}}$ is a
generally stronger invariant from which $\Phi_X^{\mathbb{Z}}$ can be
recovered; enhancements have been studied in \cite{BN, CJKLS, HN, NW} to name
just a few.

In \cite{NR} the first and last listed authors introduced the notion of
\textit{quantum enhancements} of $\Phi_X^{\mathbb{Z}}$ defined as quantum
invariants of biquandle-colored knot or link diagrams, focusing on the
unoriented case. In this paper we introduce a new infinite family of quantum 
enhancements using 
\textit{biquandle brackets}, i.e., skein relations which depend on biquandle
colorings. This family of invariants includes biquandle counting invariants,
biquandle (and quandle) cocycle invariants, and classical quantum invariants
such as the Jones and HOMFLYPT polynomials (see for example \cite{L}) as 
special cases. In particular,
we provide examples of \textit{strongly heterogeneous} quantum enhancements,
i.e., solutions to the biquandle-colored Yang-Baxter equation which are not
solutions to the uncolored Yang-Baxter equation, settling a question from 
\cite{NR} and confirming that there are quantum enhancements which are neither
cocycle invariants nor classical skein invariants.

The biquandle bracket conditions we find are very similar to the biquandle 
2-cocycle condition, and indeed biquandle 2-cocycle invariants form a special 
case of biquandle brackets. Moreover, we identify an equivalence relation
on biquandle brackets yielding the same invariant which specializes to the
cohomology relation for biquandle cocycles, even for non-cocycle biquandle 
brackets. Connections between quantum invariants and quandle cocycle invariants
were also studied in \cite{G}.

The paper is organized as follows. In Section \ref{B} we review the basics of 
biquandles and the biquandle counting invariant. In Section \ref{E} we 
define biquandle brackets and provide some examples, including as an
application a new skein invariant with values in the Galois field of eight
elements $\mathbb{F}_8$. In Section \ref{Qb} we consider the special case of 
biquandle brackets when $X$ is a quandle. We end in Section \ref{Q} with some 
open questions for future research. 

\section{\large\textbf{Biquandles}}\label{B}

A \textit{biquandle} is a set $X$ with two binary operations 
$\utr,\otr:X\times X\to X$ satisfying for all $x,y,z\in X$
\begin{itemize}
\item[(i)] $x\utr x=x\otr x$,
\item[(ii)] the maps $\alpha_y,\beta_y:X\to X$ and $S:X\times X\to X\times X$
defined by $\alpha_y(x)=x\otr y$, $\beta_y(x)=x\utr y$ and $S(x,y)=(y\otr x,x\utr y)$ are invertible, and
\item[(iii)] the \textit{exchange laws} are satisfied:
\[\begin{array}{rcl}
(x\utr y)\utr(z\utr y) & = & (x\utr z)\utr (y\otr z) \\
(x\utr y)\otr(z\utr y) & = & (x\otr z)\utr (y\otr z) \\
(x\otr y)\otr(z\otr y) & = & (x\otr z)\otr (y\utr z). 
\end{array}\]
If $x\otr y=x$ for all $x,y\in X$, we say $X$ is a \textit{quandle}.
\end{itemize}
If $X$ and $Y$ are biquandles, then a \textit{biquandle homomorphism} is a map
$f:X\to Y$ such that for all $x,y\in X$, we have
\[f(x\utr y)=f(x)\utr f(y)\quad\mathrm{and}\quad f(x\otr y)=f(x)\otr f(y).\]

The biquandle axioms come from the \textit{Reidemeister moves} where we 
interpret 
$x\utr y$ as $x$ crossing under $y$ and $y\otr x$ as $y$ crossing over $x$
from left to right when the crossing has both strands oriented down as shown.
\[\includegraphics{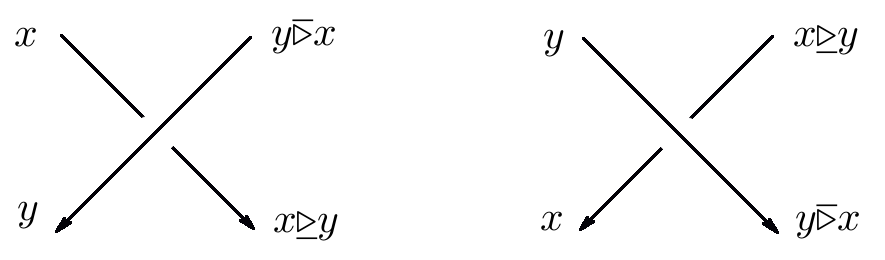}\]

\begin{remark}
Note that there are four oriented Reidemeister I moves, four oriented 
Reidemeister II moves, and eight oriented Reidemeister III moves. In \cite{P} 
several generating sets of oriented Reidemeister moves are identified; by 
Theorem  1.2 of \cite{P}, the set of moves including all four Type I moves, all four type II moves and the single type III move with all positive crossings is a 
generating set of oriented Reidemeister moves. In particular, for biquandles 
and the biquandle brackets in the next section, we will consider only these
nine moves. 
\end{remark}

\begin{remark}
We are using the notation for biquandles from \cite{EN}; note that in the 
literature, particularly in older papers, it was more common to use the 
``downward'' operations rather than our ``sideways'' operations. The newer 
notation is preferable for several reasons: the axioms are more symmetric 
and easier to remember, and the boundary map in biquandle homology is much 
simpler with this notation. See \cite{EN} for more details and further 
discussion.
\end{remark}

Then the biquandle axioms are the conditions required for every valid biquandle
coloring of the semiarcs in a knot diagram before a move to correspond to a 
unique valid biquandle coloring  (i.e., coloring satisfying the condition 
pictured above at every crossing) of the diagram after the move. All four 
oriented type I moves require that $x\utr x=x\otr x$.
\[\includegraphics{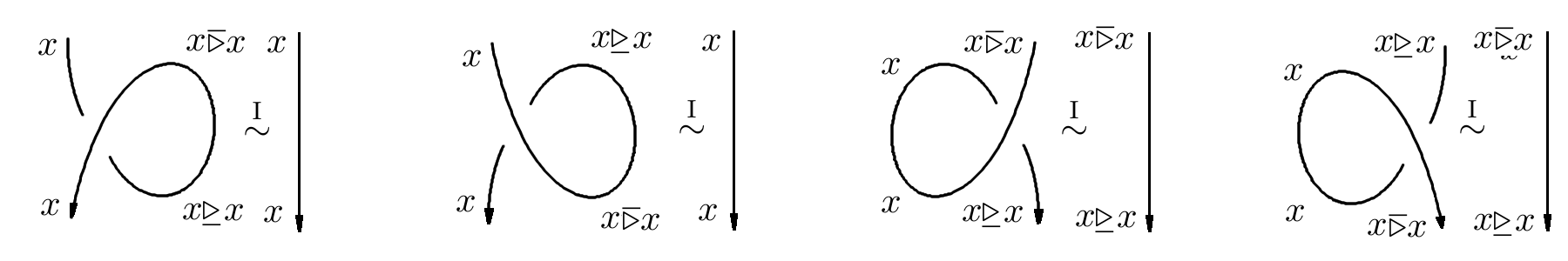}\]
The \textit{direct} type II moves, in which the strands are oriented
in the same direction, require that $y\otr x$ and $x\utr y$ are 
right-invertible.
\[\includegraphics{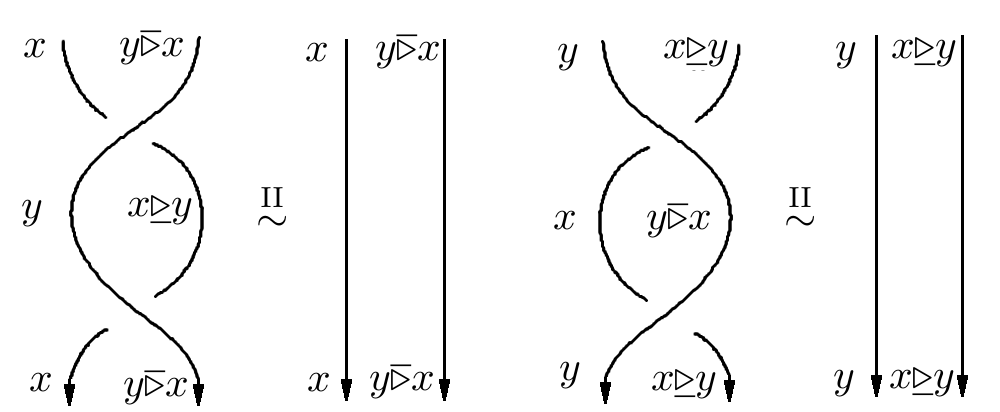}\]
The \text{reverse} type II moves, in which the strands are oriented in
opposite directions, require the map $(x,y)\mapsto(y\otr x,x\utr y)$ to be
invertible.
\[\includegraphics{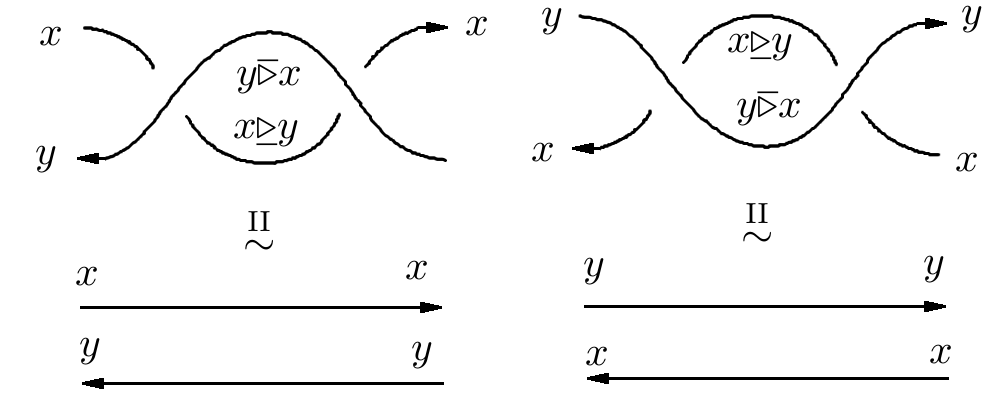}\]
Finally, the exchange laws result from the Reidemeister III move.
\[\includegraphics{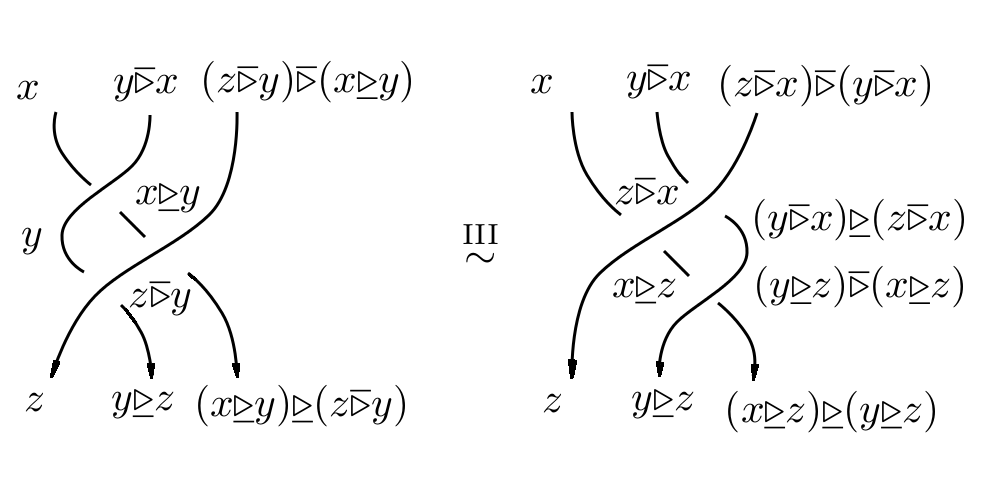}\]

\begin{example}
Let $X$ be any set and $\sigma:X\to X$ any bijection. Then $X$ is a biquandle
with operations
\[x\utr y=x\otr y=\sigma(x)\]
known as a \textit{constant action biquandle}. If $\sigma$ is the identity,
then $X$ is a \textit{trivial quandle}.
\end{example}

\begin{example}
Let $\ddot\Lambda=\mathbb{Z}[t^{\pm 1},r^{\pm 1}]$. Then any $\ddot\Lambda$-module
$A$ is a biquandle, known as an \textit{Alexander biquandle}, under the 
operations
\[x\utr y=tx+(r^{-1}-t)y\quad\mathrm{and}\quad x\otr y= r^{-1}y.\]
In particular, any commutative ring $A$ becomes an Alexander biquandle with a 
choice of invertible elements $t,r\in A$.
\end{example}

We can express the biquandle operations on a set $X=\{x_1,\dots, x_n\}$ with
operation tables for $\utr$ and $\otr$ expressed as an $n\times 2n$ block
matrix such that the entries in row $k$ columns $j$ and $n+j$ are the
subscripts of $x_k\utr x_j$ and $x_k\otr x_j$ respectively.

\begin{example}\label{bex1}
The Alexander biquandle structure on $\mathbb{Z}_5=\{1,2,3,4,5\}$
(where $5$ represents the class of $0$ so our block rows and columns
are numbered $1$ through $5$) with $t=2$ and $r=4$ can be expressed as 
the block matrix
\[\left[\begin{array}{rrrrr|rrrrr}
4 & 1 & 3 & 5 & 2 & 4 & 4 & 4 & 4 & 4 \\
1 & 3 & 5 & 2 & 4 & 3 & 3 & 3 & 3 & 3 \\
3 & 5 & 2 & 4 & 1 & 2 & 2 & 2 & 2 & 2 \\
5 & 2 & 4 & 1 & 3 & 1 & 1 & 1 & 1 & 1 \\
2 & 4 & 1 & 3 & 5 & 5 & 5 & 5 & 5 & 5
\end{array}\right].\]
\end{example}

\begin{example}\label{extref}
Let $L$ be a tame oriented knot or link. The \textit{fundamental biquandle} of
$L$, denoted $\mathcal{B}(L)$, is the set of equivalence classes of biquandle
words in a set of generators corresponding with the semiarcs in a diagram of 
$L$ under the equivalence relation generated by the crossing relations of
$L$ and the biquandle axioms. For instance, the trefoil knot $3_1$ 
\[\includegraphics{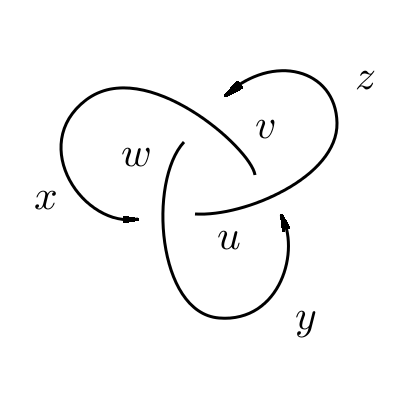}\]
has the fundamental biquandle presentation
\[\mathcal{B}(3_1)=\langle x,y,z,u,v,w \ |\ x\utr y=u, y\otr x= w, y\utr z=v, 
z\otr y=u, z\utr x=w, x\otr z=v\rangle.\]
Then for instance in $\mathcal{B}(3_1)$ we have
\[ (y\otr u)\otr (x\otr u) 
 =  (y\otr x)\otr (u\utr x) 
 =  w\otr (u\utr x).
\]
Different diagrams of the same knot or link yield different presentations 
which differ by Tietze moves and hence present the same biquandle.
\end{example}

Given a finite biquandle $X$ and a tame knot or link diagram 
$L$, a \textit{biquandle coloring} of $L$ is an assignment of elements of
$X$ to the semiarcs in $L$ such that the crossing relations
\[\includegraphics{sn-mo-vr-10.png}\]
are satisfied at every crossing. Such an assignment determines and is 
determined by a 
biquandle homomorphism $f:\mathcal{B}(L)\to X$. In particular, the set of
biquandle colorings of $L$ can be identified with the set 
$\mathrm{Hom}(\mathcal{B}(L),X)$
of biquandle homomorphisms from the fundamental biquandle of $L$ to $X$.
If $L$ is tame, then $\mathcal{B}$ is finitely generated with $2n$ generators
where $n$ is the number of semiarcs in $L$; hence 
$|\mathrm{Hom}(\mathcal{B}(L),X)| \le |X|^{2n}$. We usually write
$|\mathrm{Hom}(\mathcal{B}(L),X)|=\Phi_X^{\mathbb{Z}}(L)\in\mathbb{N}$;
this cardinality is known as the \textit{biquandle counting invariant} 
\cite{BN}.

\begin{example}
The figure 8 knot $4_1$ below has only five valid biquandle
coloring by the Alexander biquandle in example \ref{bex1}, 
as can be determined
by row-reducing over $\mathbb{Z}_5$ the coefficient matrix of the system of
crossing equations or by brute-force checking all possible colorings
and counting those which satisfy the crossing relations. These colorings are 
pictured below:
\[\includegraphics{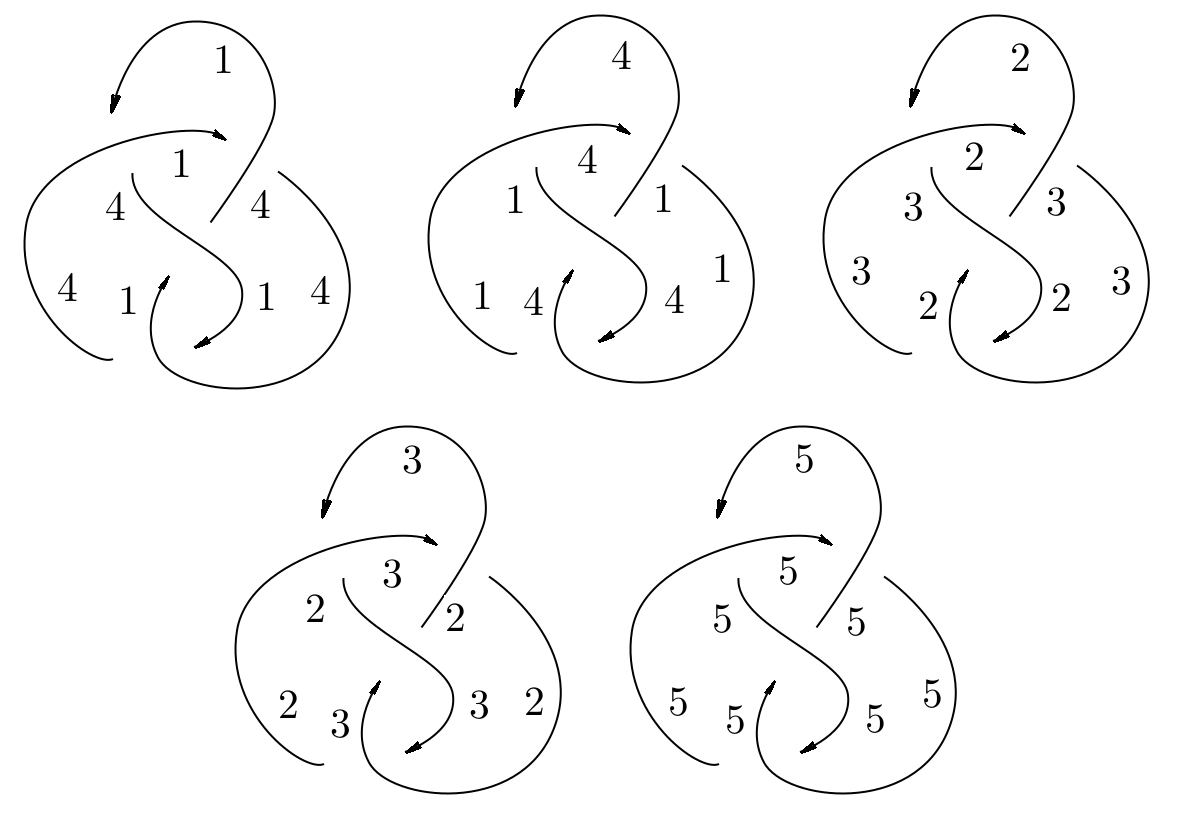}\]
\end{example}

\section{\large\textbf{Biquandle Brackets}}\label{E}

We would like to define a skein invariant (see \cite{L} for instance) for 
biquandle-labeled link diagrams. Let $X$ be a finite biquandle, and let us 
fix a commutative ring with identity $R$ and denote the set of units of $R$ 
as $R^{\times}$. We would like to choose elements $A_{x,y},B_{x,y},w\in R^{\times}$ 
and $\delta\in R$ such that the element of $R$ determined by the skein 
relations
\[\includegraphics{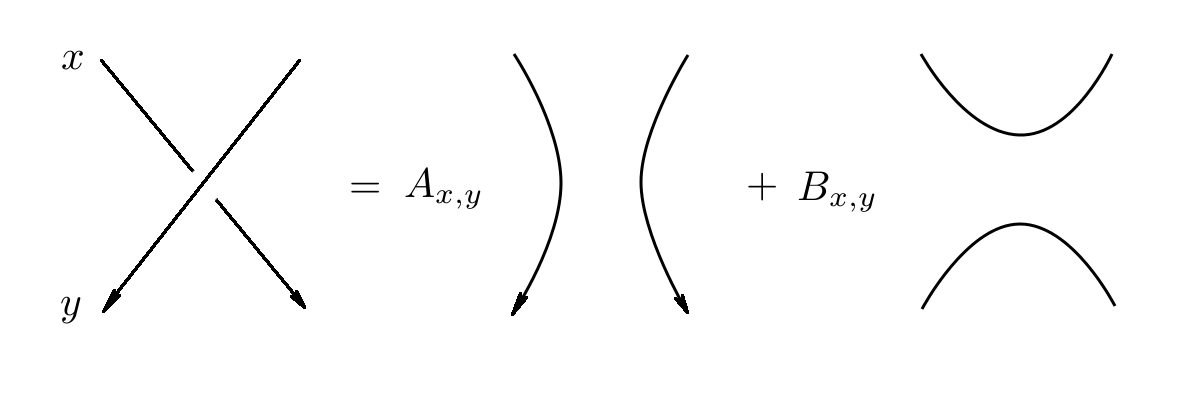}\]
\[\includegraphics{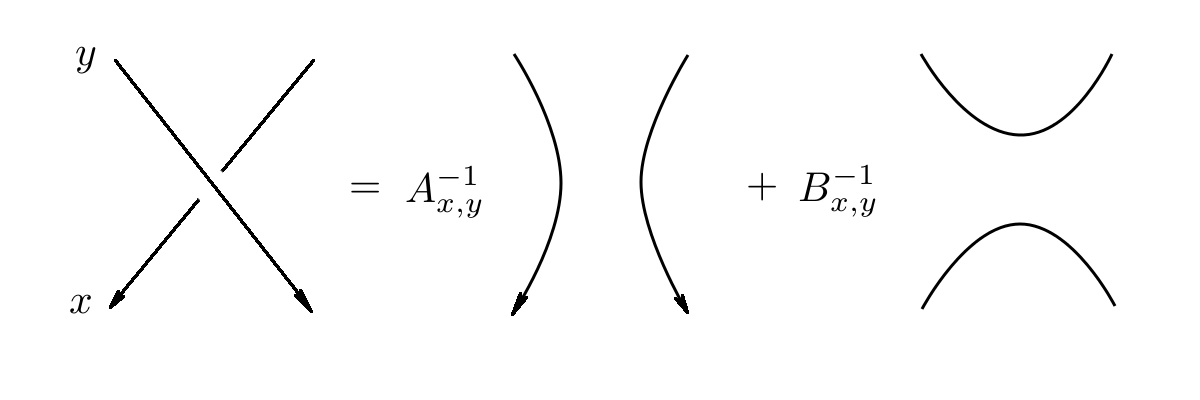}\]
with $\delta$ the value of a simple closed curve and $w$ the value of a 
positive kink is an invariant of $X$-labeled Reidemeister moves. For a given 
choice of $X$ and $R$, we will denote  such a collection of 
$A_{x,y}, B_{x,y},w$ and $\delta$ by $\beta$.

More precisely, for each $X$-coloring $f$ of an oriented link diagram $D$ 
with $c$ crossings, we 
will find the collection of $2^c$ Kauffman states obtained by smoothing all the 
crossings as depicted above; to each such state will be associated a value
in $R$ consisting of the product of the smoothing coefficients at crossings
with $w^{n-p}\delta^k$ where $k$ is the number of circles in the Kauffman state,
$n$ is the number of negative crossings and $p$ is the number of positive
crossings in $D$. The sum of these contributions from each state will be 
denoted $\beta(f)$. The multiset of $\beta(f)$-values over the set of 
$X$-colorings will then be an invariant of oriented knots and links, which we 
will denote by $\Phi_X^{\beta,M}$; here the subscript $X$ indicates the 
coloring biquandle, the $\beta$ specifies the enhancement of the $X$-counting
invariant, and the $M$ indicates the multiset version of the invariant. 

\begin{remark}
It is standard practice for enhancements of counting invariants to be 
expressed in ``polynomial form'' by writing elements of the multiset as 
exponents of a formal variable $u$ with positive integer multiplicities as 
coefficients. We note that while strictly speaking this only defines a genuine 
(Laurent) polynomial in case $R=\mathbb{Z}$, this notation in common in the 
literature -- it was introduced with quandle cocycle invariants in \cite{CJKLS} 
and has been standard ever since, see \cite{BN,CES,CEGN,EN,G,HN,NR,NW} for 
instance.
The invariant written in this format contains the same information as the 
multiset version and has certain advantages; for instance, evaluation of
$\Phi^{\beta}_X$ at $u=1$ (using the rule $1^r=1$ for all $r\in R$) yields the 
cardinality of the multiset version of the invariant , i.e. the 
$X$-counting invariant:
\[\Phi_X^{\mathbb{Z}}(K) =\left.\Phi_X^{\beta}(K)\right|_{u=1}.\]
Moreover, for certain brackets (as we will see later) using
this format enables a factorization of the biquandle bracket polynomial
as a product of a specialization of the Kauffman bracket polynomial
with a biquandle cocycle polynomial. Additionally, the reader may find it 
easier to compare polynomials at a glance than to compare multisets.
Of course, if preferred one can always use the multiset notation; we will
generally use the polynomial form when $R=\mathbb{Z}$ or $\mathbb{Z}_n$ and
the multiset form otherwise.
\end{remark}

%\begin{remark}
Let us address the obvious objection right away: smoothing a crossing in an
$X$-labeled oriented link diagram does indeed result in diagrams without 
coherent biquandle colorings or even orientations. It follows 
that some modifications are needed if we wish to define the invariant 
recursively, as is often done with skein invariants; this will be a topic for
another paper.
%see remark \ref{rem:2}
%in the next section for a suggestion of how to do this. 
However, we note that such a recursive definition
is not necessary, since we can instead simply define the invariant via the 
state-sum formulation, i.e., define a function on oriented $X$-colored link
diagrams by the summing the products of crossing coefficients times 
appropriate powers of $\delta$ and $w$ 
over the set of all completely smoothed states, 
with invariance following from the fact that this value is unchanged 
by $X$-colored Reidemeister moves. We can conceptualize this state-sum 
method as moving the biquandle colors off the crossings and into the 
coefficients for each state.
%\end{remark}

Recall that the set of four oriented Reidemeister I moves, four oriented 
Reidemeister II moves and the single Reidemeister III move with all positive 
crossings forms a generating set of oriented Reidemeister moves \cite{P}. 
Moreover, we have the following observation:

\begin{observation}
In a state sum defined as a sum over the set of states of the 
product of crossing weights times $\delta^k$ where $k$ is the number of 
circles in a state, local moves which preserve boundary connectivity,
number of circles in each state and local crossing weight (i.e. 
coefficient product) do not change the state sum.
\end{observation}

Thus, we can find the conditions a biquandle bracket $\beta$ must satisfy to
define an invariant by identifying conditions such that the local crossing 
weights are preserved by the above-identified list of Reidemeister moves.

The first Reidemeister move comes in four oriented versions; the two positively 
oriented moves require that for all $x\in X$, we have $A_{x,x}\delta+B_{x,x}=w$,
while the negatively oriented moves require $A_{x,x}^{-1}\delta+B_{x,x}^{-1}=w^{-1}$.
In particular, we can think of writhe-reducing type I moves as factoring out 
powers of $w$ and writhe-increasing type I moves as factoring out powers
of $w^{-1}$.
\[\includegraphics{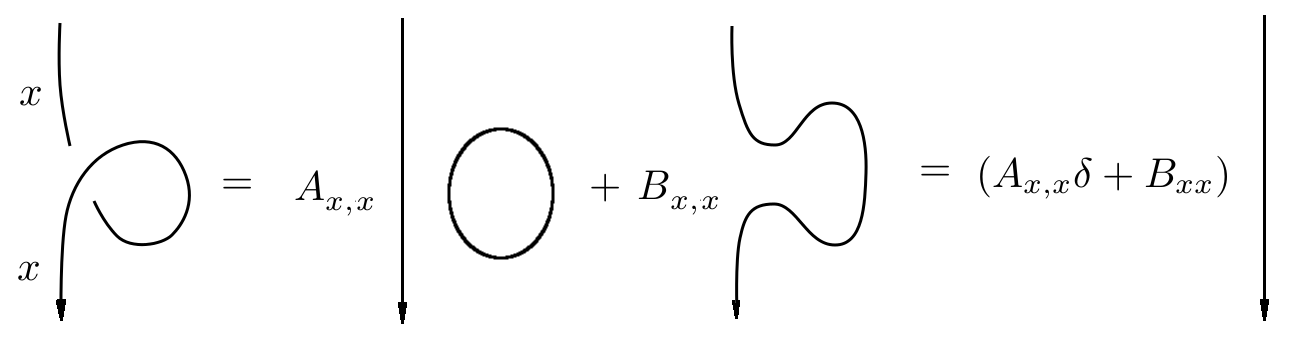}\]
The direct type II moves require the oriented smoothing coefficients at 
positive and negative crossings to be multiplicative inverses, with the reverse
II moves requiring the same of the unoriented smoothing coefficients; all four 
moves then require that $\delta=-A_{x,y}B_{x,y}^{-1}-A_{x,y}^{-1}B_{x,y}$.
\[\includegraphics{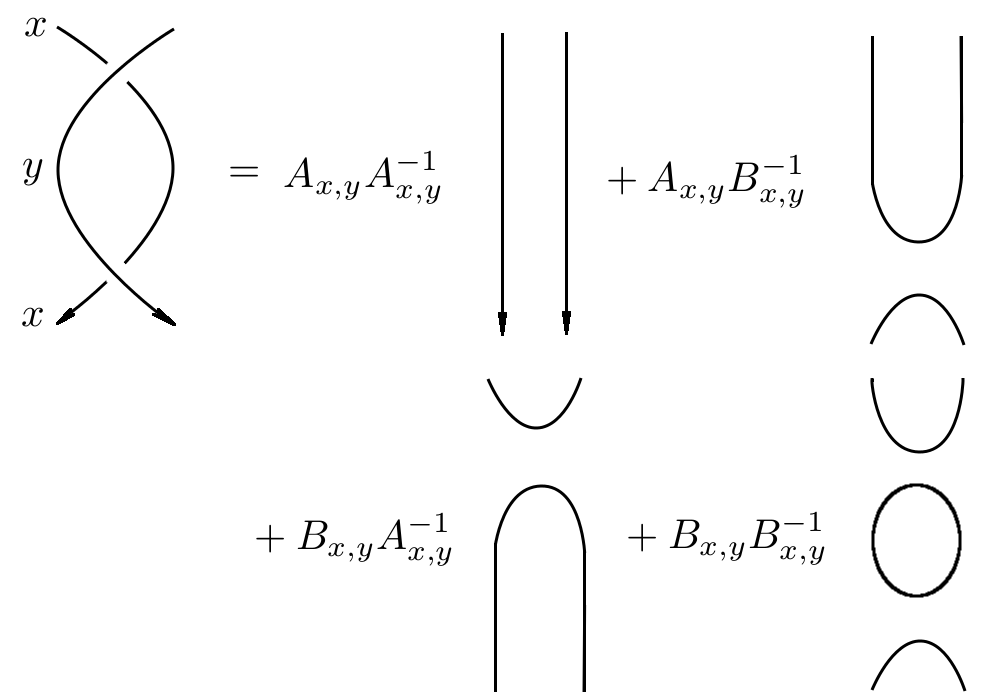}\]
\[\includegraphics{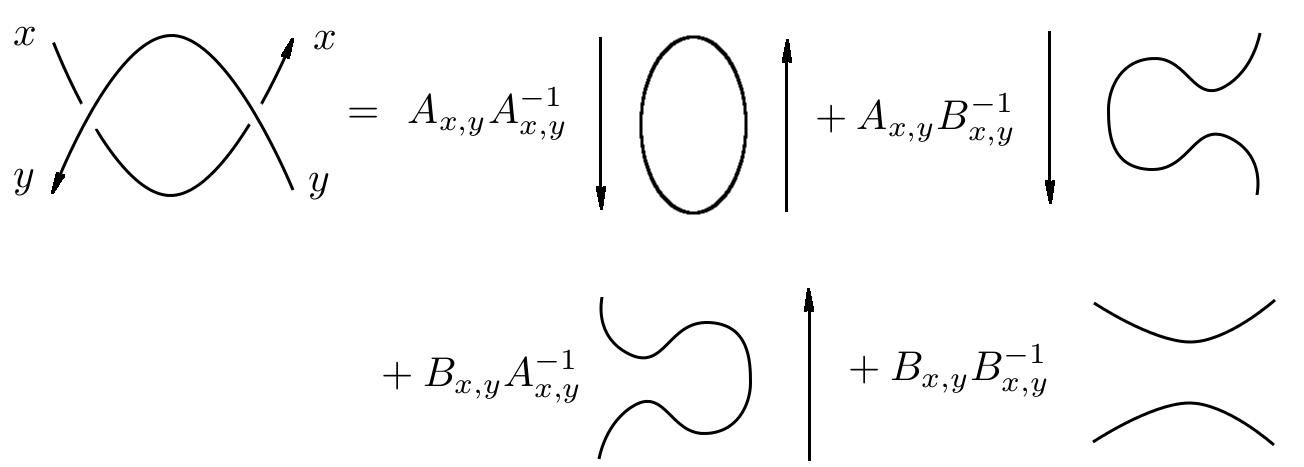}\]

Comparing coefficients of the five crossingless diagrams 
\[\includegraphics{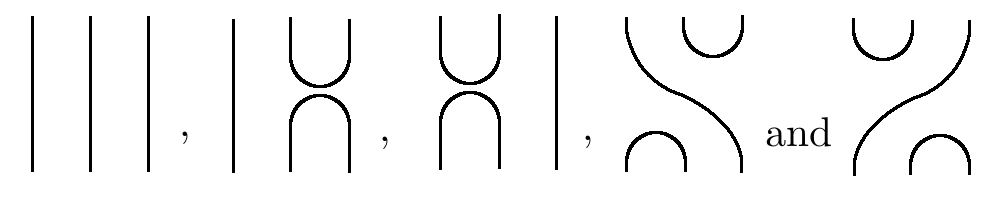}\]
on both sides of the $X$-labeled Reidemeister III move, we have on the left side
\[\includegraphics{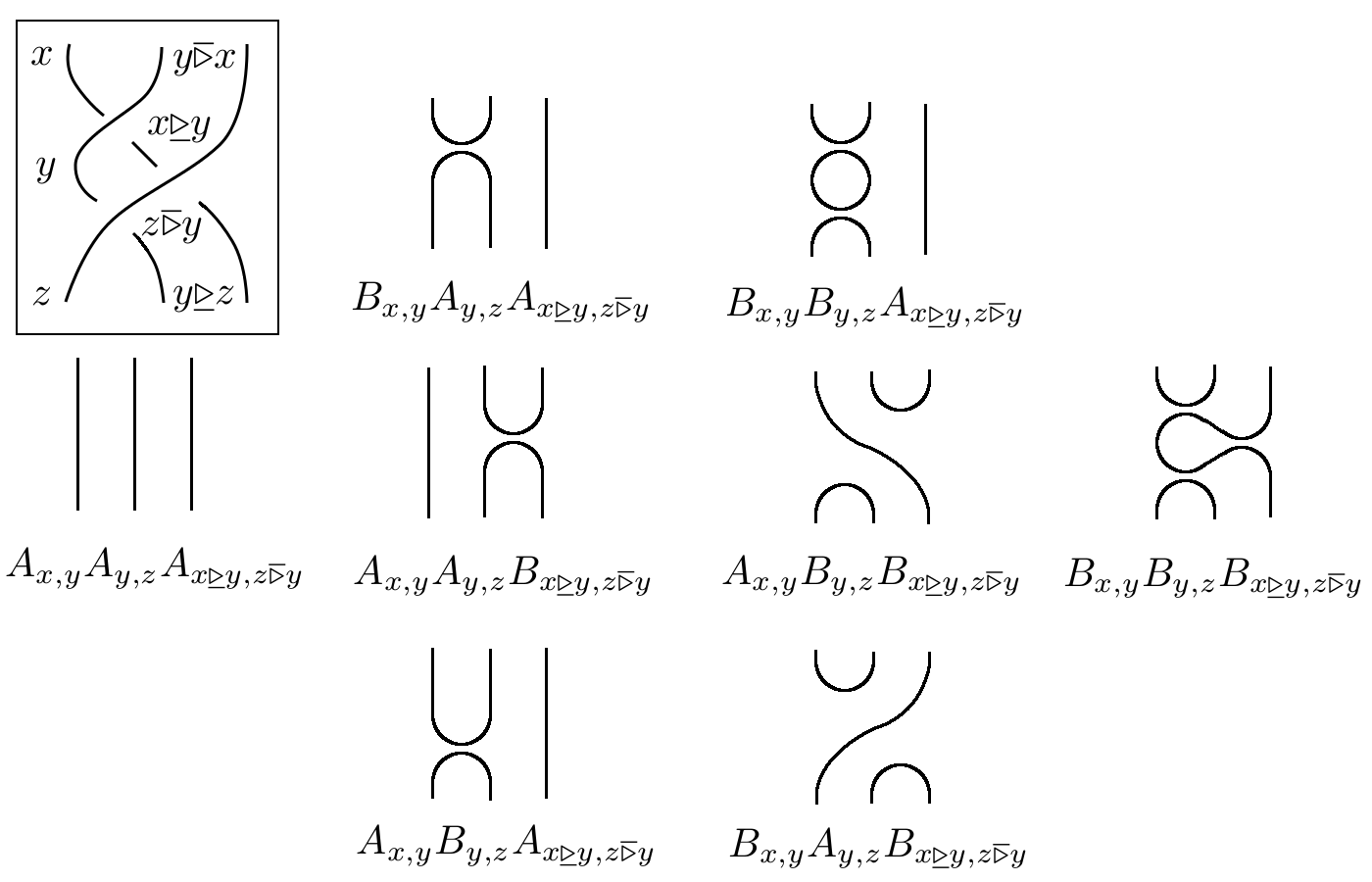}\]
and on the right side
\[\includegraphics{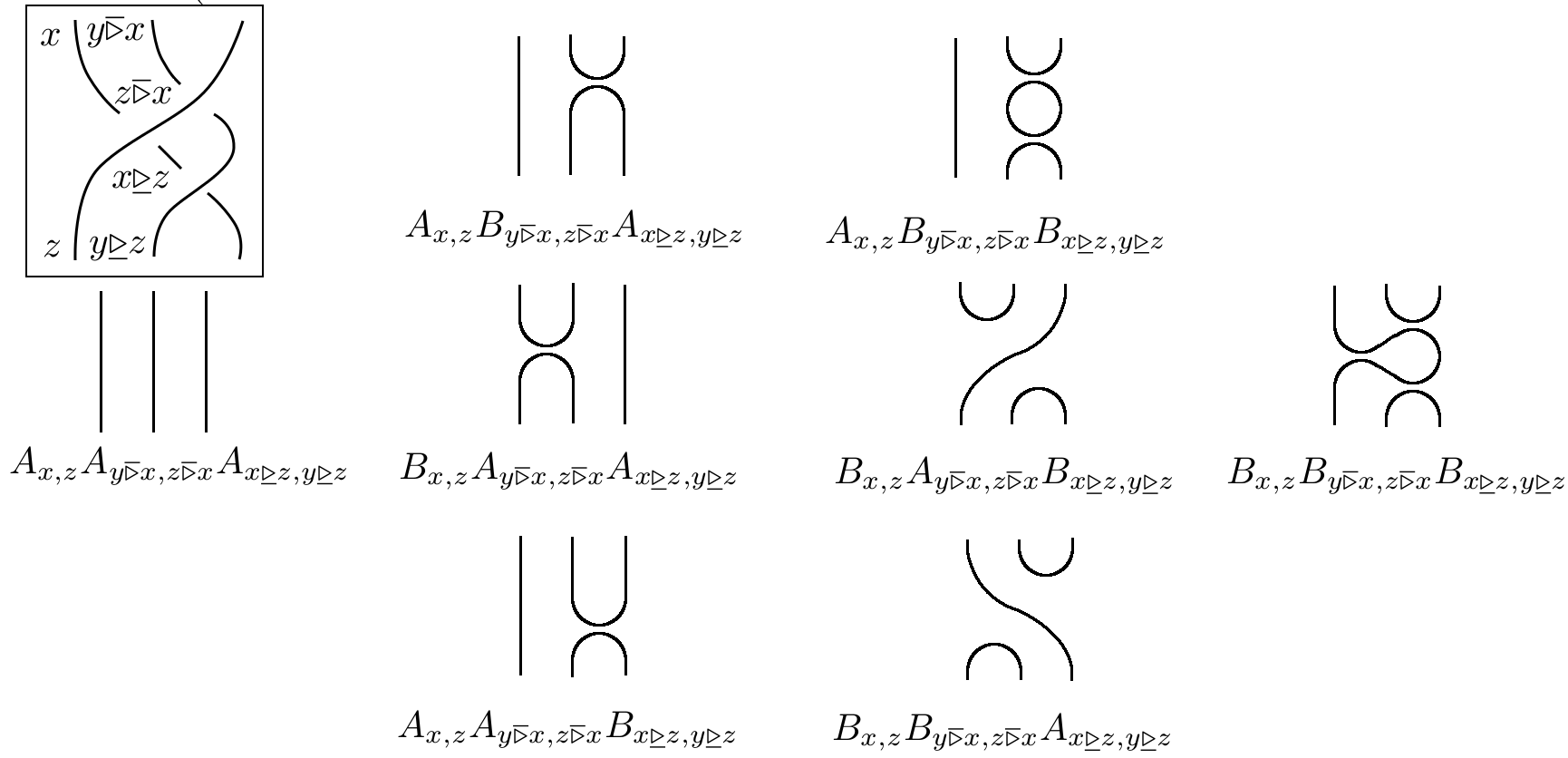}\]
yielding the remaining conditions on $A_{x,y}$ and $B_{x,y}$.
We thus obtain
\begin{definition}
Let $X$ be a finite biquandle and $R$ be a commutative ring with identity. A 
\textit{biquandle bracket} on $X$ with values in $R$, also called an 
\textit{$X$-bracket}, is a pair of maps $A,B:X\times X\to R^{\times}$ and 
distinguished elements $\delta\in R$ and $w\in R^{\times}$ satisfying
\begin{itemize}
\item[(i)] for all $x\in X$,
\[\delta A_{x,x}+B_{x,x}=w\quad \mathrm{and}\quad \delta A_{x,x}^{-1}+B_{x,x}^{-1}=w^{-1}\]
\item[(ii)] for all $x,y\in X$,
\[\delta=-A_{x,y}B_{x,y}^{-1}-A_{x,y}^{-1}B_{x,y}\]
and 
\item[(iii)] for all $x,y,z\in X$, 
\[\begin{array}{rcl}
A_{x,y}A_{y,z}A_{x\utr y,z\otr y} & = & A_{x,z}A_{y\otr x,z\otr x}A_{x\utr z,y\utr z} \\
A_{x,y}B_{y,z}B_{x\utr y,z\otr y} & = & B_{x,z}B_{y\otr x,z\otr x}A_{x\utr z,y\utr z} \\
B_{x,y}A_{y,z}B_{x\utr y,z\otr y} & = & B_{x,z}A_{y\otr x,z\otr x}B_{x\utr z,y\utr z} \\
A_{x,y}A_{y,z}B_{x\utr y,z\otr y} & = & 
A_{x,z}B_{y\otr x,z\otr x}A_{x\utr z,y\utr z} 
+A_{x,z}A_{y\otr x,z\otr x}B_{x\utr z,y\utr z} \\ 
& & +\delta A_{x,z}B_{y\otr x,z\otr x}B_{x\utr z,y\utr z} 
+B_{x,z}B_{y\otr x,z\otr x}B_{x\utr z,y\utr z} \\
B_{x,y}A_{y,z}A_{x\utr y,z\otr y} 
+A_{x,y}B_{y,z}A_{x\utr y,z\otr y} & & \\
+\delta B_{x,y}B_{y,z}A_{x\utr y,z\otr y} 
+B_{x,y}B_{y,z}B_{x\utr y,z\otr y}  
& = & B_{x,z}A_{y\otr x,z\otr x}A_{x\utr z,y\utr z} \\
\end{array}\]
\end{itemize}
where $A(x,y)$ and $B(x,y)$ are denoted $A_{x,y}$ and $B_{x,y}$.
\end{definition}

Given a finite biquandle $X=\{x_1,\dots,x_n\}$, an $X$-bracket can be 
represented by a pair of $n\times n$ matrices $A,B$ with $A_{j,k}=A(x_j,x_k)$
and $B_{j,k}=B(x_j,x_k)$. We will usually write these as a single $n\times 2n$
block matrix for convenience. Note that we can recover $\delta$ and $w$
from such a matrix, with 
\[\delta=-A_{1,1}B_{1,1}^{-1}-A_{1,1}^{-1}B_{1,1}
\quad \mathrm{and}\quad w=A_{1,1}\delta+B_{1,1}.\]

\begin{example}
Let $X=\{1\}$ be the biquandle with one element. Then the matrix 
\[\left[\begin{array}{r|r}
A & A^{-1}
\end{array}\right]\]
where $A\in\mathbb{Z}[A^{\pm 1}]$ is an invertible variable
defines a biquandle bracket with
\[\delta=-A(A^{-1})^{-1}-(A^{-1})A^{-1}=-A^2-A^{-2}
\quad \mathrm{and}\quad w=A(-A^2-A^{-2})+A^{-1}=-A^3.\]
Indeed, this is the Kauffman bracket (see for example \cite{L,S}).
\end{example}

\begin{example}\label{ex:coboundary}
Let $X$ be a finite  biquandle, $R$ be a commutative ring, and $C:X\to R^{\times}$
be a map where we write $C_x$ for $C(x)$. Then the maps 
$A,B:X\times X\to R^{\times}$ defined by
\[A_{x,y}=B_{x,y}=C_xC_y^{-1}C_{x\utr y}^{-1}C_{y\otr x}\]
for all $x,y\in X$ define a biquandle bracket with $\delta=-2$ and $w=-1$.
To see this, we note that if $A_{x,y}=B_{x,y}$, we necessarily have $\delta=-2$,
and biquandle bracket axiom (iii)'s five equations all reduce to the first
equation, namely
\[A_{x,y}A_{y,z}A_{x\utr y, z\otr y}  = A_{x,z}A_{y\otr x,z\otr x}A_{x\utr z, y\utr z}.\]
Then
\begin{eqnarray*}
A_{x,y}A_{y,z}A_{x\utr y, z\otr y} & = & 
(C_xC_y^{-1}C_{x\utr y}^{-1}C_{y\otr x})
(C_yC_z^{-1}C_{y\utr z}^{-1}C_{z\otr y})
(C_{x\utr y}C_{z\otr y}^{-1}C_{(x\utr y)\utr (z\otr y)}^{-1}C_{(z\otr y)\otr (x\utr y)})\\
& = & 
C_xC_{y\otr x}
C_z^{-1}C_{y\utr z}^{-1}
C_{(x\utr y)\utr (z\otr y)}^{-1}C_{(z\otr y)\otr (x\utr y)}
\end{eqnarray*}
while
\begin{eqnarray*}
A_{x,z}A_{y\otr x,z\otr x}A_{x\utr z, y\utr z} 
& = &
(C_xC_z^{-1}C_{x\utr z}^{-1}C_{z\otr x})
(C_{y\otr x}C_{z\otr x}^{-1}C_{(y\otr x)\utr(z\otr x)}^{-1}C_{(z\otr x)\otr(y\otr x)})C_{x\utr z}) \\
& & \quad \quad \quad \times\  (C_{y\utr z}^{-1}C_{(x\utr z)\utr(y\utr z)}^{-1}C_{(y\utr z)\otr(x\utr z)})
\\
& = &
C_xC_z^{-1}
C_{y\otr x}C_{(z\otr x)\otr(y\otr x)}
C_{y\utr z}^{-1}C_{(x\utr z)\utr(y\utr z)}^{-1}
\end{eqnarray*}
which are equal by the exchange laws.  As we will later see, this type of 
biquandle bracket is actually a biquandle 2-coboundary; in particular, 
two biquandle brackets which differ by a coboundary are ``cohomologous''
and define the same invariant.
\end{example}

We now introduce the first of our new invariants.

\begin{definition}
Let $L$ be an oriented knot or link diagram with $n$ crossings with
generators $x_1,\dots, x_{2n}$ for the fundamental biquandle $\mathcal{B}(L)$
associated to the semiarcs. There are $2^n$ \textit{states} corresponding to
choices of oriented or unoriented smoothing for each crossing, each of
which has an associated product of $n$ factors of $A_{x,y}^{\pm 1}$ or 
$B_{x,y}^{\pm 1}$
times $\delta^k$ where $k$ is the number of components of the state. The 
sum of these contributions times the \textit{writhe correction factor},
$w^{n-p}$, is the \textit{fundamental biquandle bracket value} for $L$.
\end{definition}

\begin{example} The Hopf link $L2a1$ below has four smoothed states with 
coefficients as listed.
\[\includegraphics{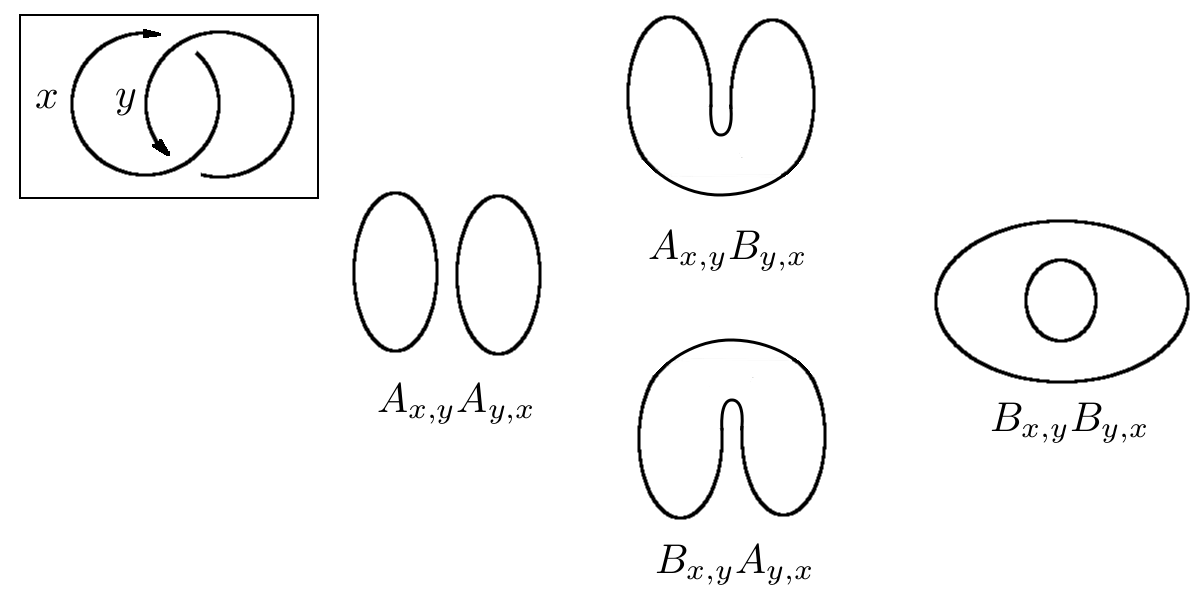}\] 
Then the Hopf link has fundamental biquandle bracket value
\[\phi=w^{-2}(A_{x,y}A_{y,x}\delta^2+B_{x,y}A_{y,x}\delta+A_{x,y}B_{y,x}\delta
+B_{x,y}B_{y,x}\delta^2)\]
where $x,y$ are generators of the fundamental biquandle 
$\mathcal{B}(L2a1)=\langle x,y \ |\ x\utr y=x\otr y,\ y\utr x=y\otr x \rangle$.
\end{example}

The fundamental biquandle bracket treats every knot or link as colored by 
elements of its fundamental biquandle. This fundamental biquandle bracket 
may be a complete invariant of %virtual 
knots since it includes
the fundamental biquandle and hence the fundamental quandle, already known %conjectured 
to be a complete invariant for virtual knots %links 
up to a type of reflection \cite{HK}, and our later
examples demonstrate that the fundamental biquandle bracket can detect mirror 
images. However, comparing fundamental biquandle bracket values for different
knots and links is not straightforward since any two such links are being 
colored by generally different biquandles.

To get a more immediately useful invariant, let $X$ be a finite biquandle.
For any $X$-bracket $\beta$ over $R$, evaluating the fundamental biquandle
bracket value of an $X$-coloring $f$ of an oriented link diagram $L$ yields an 
element of $R$ which is unchanged by $X$-colored Reidemeister moves on $L$; 
let us denote this value by $\beta(f)$.

\begin{definition} 
Let $X$ be a finite biquandle, $L$ an oriented link and $\beta$ a biquandle 
bracket. Then the \textit{biquandle bracket multiset} invariant of $L$ is
the multiset of $\beta$-values over the set of $X$-labelings of $L$,
\[\Phi_X^{\beta,M}(L)=\{\beta(f)\ |\ f\in\mathrm{Hom}(\mathcal{B}(L),X)\}\]
and the \textit{biquandle bracket polynomial} invariant of $L$ is
\[\Phi_X^{\beta}(L)=\sum_{f\in\mathrm{Hom}(\mathcal{B}(L),X)} u^{\beta(f)}.\]
\end{definition}

\begin{remark}
Again, note that in this ``polynomial'' form, the coefficients are integers, 
$u$ is a formal variable and the exponents of $u$ are elements of $R$.
\end{remark}

\begin{proposition}
Let $X$ be a finite biquandle and let $\beta$ and $\beta'$ be $X$-brackets 
over $R$ defined by maps $A,B:X\times X\to R^{\times}$ and
$A',B':X\times X\to R^{\times}$ respectively. If there is an invertible 
scalar $\alpha\in R^{\times}$ such that for all $x,y\in X$ we have
\[A_{x,y}=\alpha A_{x,y}' \quad\mathrm{and}\quad B_{x,y}=\alpha B_{x,y}' \]
then the link invariants defined by $\beta$ and $\beta'$ are equal.
\end{proposition}

\begin{proof}
First, we note that 
\[\delta'
=-A_{x,y}'B_{x,y}'^{-1}-A_{x,y}'^{-1}B_{x,y}'
=-(\alpha A_{x,y})(\alpha B_{x,y})^{-1}-(\alpha A_{x,y}^{-1})\alpha B_{x,y}
=-A_{x,y}B_{x,y}^{-1}-A_{x,y}^{-1}B_{x,y}
=\delta
\]
and
\[w'=A_{x,x}'\delta +B_{x,x}'=\alpha A_{x,x}\delta +\alpha B_{x,x}=\alpha w.\]
Then for any link diagram $L$ with $j$ positive crossings and $k$ negative
crossings, the state sum with $\beta'$ equals that with $\beta$ multiplied 
by $\alpha^{j-k}$ at every crossing. Then the contribution $\beta'(f)$ equals
$\beta(f)$ multiplied by $\alpha^{k-k}$, then multiplied by $\alpha^{k-j}$ in 
the writhe-correction factor $(w')^{k-j}$; hence, the powers of $\alpha$ cancel 
and we have $\beta(f)=\beta'(f)$, whence
$\Phi_x^{\beta}(L)=\Phi_x^{\beta'}(L)$ for all classical and virtual knots and 
links $L$.
\end{proof}

\begin{example}\label{ex1}
The simplest non-trivial biquandle is the constant action biquandle on 
$X=\{1,2\}$ with operation matrix
\[\left[\begin{array}{rr|rr}
2 & 2 & 2 & 2 \\
1 & 1 & 1 & 1
\end{array}\right].\]
The counting invariant $\Phi^{\mathbb{Z}}_X(L)$ with respect to this biquandle
is $0$ if $L$ is a virtual link with any component containing an odd number of 
crossing points and is $2^c$ where $c$ is the number of components of $L$ 
otherwise. Our \texttt{python} computations reveal %64 
biquandle bracket structures on $X$ with coefficients in $\mathbb{Z}_5$
including
\[\left[\begin{array}{rr|rr}
1 & 3 & 4 & 2 \\
4 & 1 & 1 & 4
\end{array}\right].\]
The Hopf link has four $X$-labelings and fundamental biquandle bracket value
\[\phi=A_{x,y}A_{y,x}\delta^2+B_{x,y}A_{y,x}\delta+A_{x,y}B_{y,x}\delta+B_{x,y}B_{y,x}\delta^2.\]
Then we have $\delta=2$, $w=1$ and
\[\begin{array}{cc|c}
x & y & \phi \\ \hline
1 & 1 & 1(1)(2^2)+1(4)(2)+4(1)(2)+ 4(4)(2^2) =4+3+3+4=4\\
1 & 2 & 3(4)(2^2)+2(4)(2)+3(1)(2)+ 2(1)(2^2) =3+1+1+3=3\\
2 & 1 & 4(3)(2^2)+1(3)(2)+4(2)(2)+ 1(2)(2^2) =3+1+1+3=3\\
2 & 2 & 1(1)(2^2)+1(4)(2)+4(1)(2)+ 4(4)(2^2) =4+3+3+4=4\\
\end{array}\]
or in more pictorial form,
\[\includegraphics{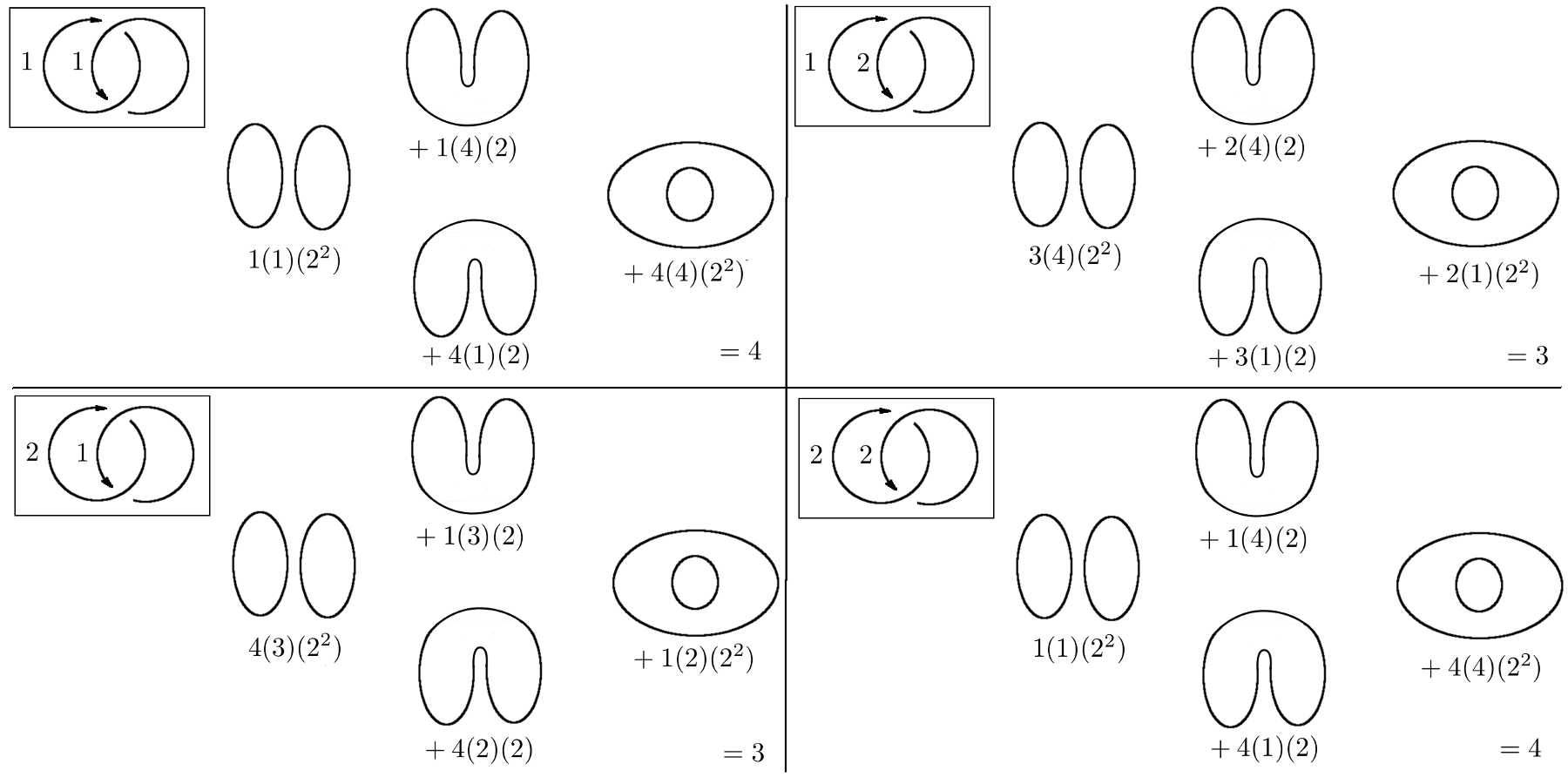}\]
Then the Hopf link has biquandle bracket invariant
\[\Phi_X^{\beta,M}(L)=\{3,3,4,4\}\]
or in ``polynomial'' form
\[\Phi_X^{\beta}(L)=2u^3+2u^4\]
while the unlink of two components $U_2$ has invariant value
\[\Phi_X^{\beta}(U_2)=4u^4.\]
\end{example}

\begin{example}
Let $X$ be any finite biquandle and $R$ be any commutative ring with identity.
For any invertible element $t\in R$, the maps $A_{x,y}=t$, $B_{x,y}=t^{-1}$ 
with $\delta=-t^{-2}-t^2$ and $w=t^3$ define a biquandle bracket 
$\beta_t$ called a \textit{constant biquandle 
bracket}. Since in this case the skein coefficients do not depend on the 
biquandle colors, each $X$-coloring gets the same state-sum value, namely
$\Phi_X^{\mathbb{Z}}(L) u^{K_L(t)}$ where $K_L(t)$ is the Kauffman bracket 
polynomial of $L$ evaluated at $t$; hence,
for any link $L$, the biquandle bracket invariant with respect to 
$\beta_t$ is $\Phi_X^{\beta}(L)=\Phi_X^{\mathbb{Z}}(L) u^{K_L(t)}$.

For example, if $R=\mathbb{Z}_7$ and $t=2$, we have $2^{-1}=4$, 
$\delta=-4^2-2^2=-16-4=-20=1$
and $w=t^3=2^3=8=1$; then in the Hopf link example above we have 
\[\includegraphics{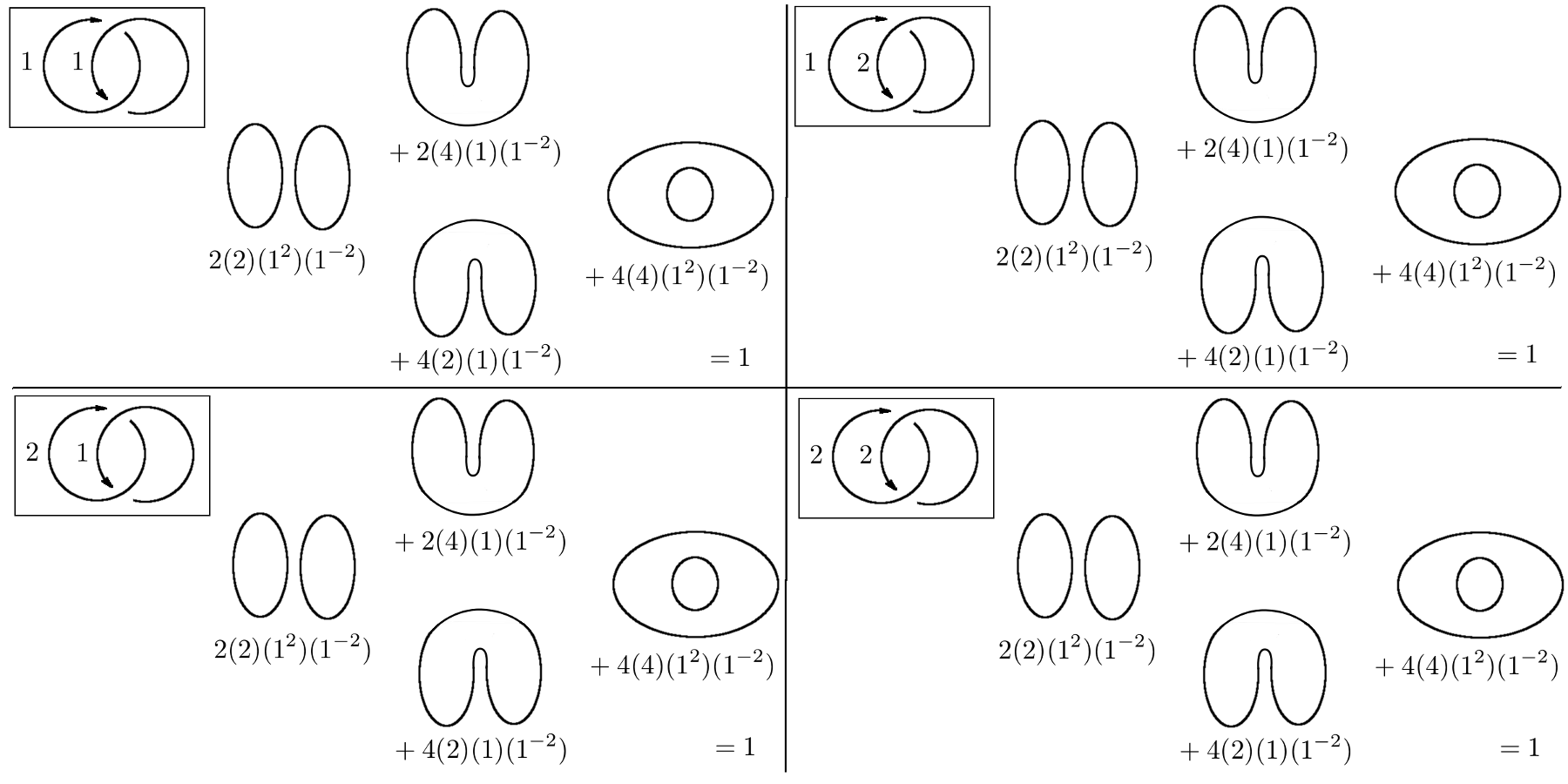}\]
so we have 
\[\Phi_X^{\beta}(L)=4u=\Phi_X^{\mathbb{Z}}(L) u^{K_L(t)}\]
since we have $\Phi_X^{\mathbb{Z}}(L)=4$ and 
$K_L(2)=\left. 1+t^{-4}+t^{-8}+t^{-12}\right|_{t=2}=1+4+2+1=1$.
\end{example}

\begin{example}
More generally, an $X$-bracket in which for all $x,y\in X$ we have 
$A_{x,y}=A$ and $B_{x,y}=B$ defines a skein invariant which does not use the 
biquandle colors; the biquandle bracket invariant will then be 
$\Phi_X^{\mathbb{Z}}(K)$ copies of the skein invariant thus defined.
For instance, if $A+A^{-1}=B+B^{-1}$, we have a biquandle bracket 
$\beta$ satisfying 
$\Phi_X^{\beta_{A,B}}(L)=\Phi_X^{\mathbb{Z}}(L) u^{K_L(-A^2B^{-1},A+A^{-1})}$ where 
$K(a,z)$ is the Kauffman 2-variable polynomial, as one can easily see 
by comparing the biquandle bracket skein relation with the usual Kauffman 
2-variable skein relation.
Similarly, a biquandle bracket $\beta$ with $A_{x,y}=\alpha A$ and 
$B_{x,y}=B$ with $\alpha^{-1}B+\alpha B^{-1}=0$ yields
$\Phi_X^{\beta_{a,z}}(L)=\Phi_X^{\mathbb{Z}}(L) u^{H_L(\alpha,\alpha^{-1}A+\alpha A^{-1})}$
where $H(a,z)$ is the HOMFLY-PT polynomial \cite{L}.
\end{example}

\begin{example}
Let $X$ be a finite biquandle, $G$ an abelian group,
and $\psi\in H^2(X;G)$ an element of the second cohomology of $X$ with
$G$ coefficients, i.e., a function $\psi:X\times X\to G$ satisfying 
for all $x,y,z\in X$
\[\psi(x,y)\psi(y,z)\psi(x\otr y,z\utr y)
=\psi(x,z)\psi(y\otr x,z\otr x)\psi(x\utr z,y\utr z)\]
and $\psi(x,x)=1$ (see \cite{CEGN} for instance). 
Then setting $A_{x,y}=B_{x,y}=\psi(x,y)$ defines a biquandle
bracket with $R=\mathbb{Z}[G]$. Indeed, every biquandle bracket with 
$A_{x,y}=B_{x,y}$ for all $x,y\in X$ arises in this way, since the biquandle 
bracket conditions with $A_{x,y}=B_{x_y}$ reduce to $\delta=-2$, $w=-1$ and
the $2$-cocycle condition
\[A_{x,y}A_{y,z}A_{x\utr y,z\otr y}  =  A_{x,z}A_{y\otr x,z\otr x}A_{x\utr z,y\utr z}.\]
The biquandle bracket invariant in this case satisfies
\[\Phi^{\beta}_X(L)=\Phi^{\psi}_X(L)K_L(1)\]
where $K_L(1)$ is the Kauffman bracket polynomial of
$L$ evaluated at $A=1$.
\end{example}

\begin{proposition}\label{cohom}
Let $X$ be a finite biquandle, $R$ be a commutative ring and $C:X\to R^{\times}$
be a map as in example \ref{ex:coboundary}, and let 
$\gamma:X\times X\to \mathbb{R}^{\times}$ be the biquandle bracket defined by 
setting both $A$ and $B$ equal to
\[\gamma(x,y)=C(x)C(y)^{-1}C(x\utr y)^{-1}C(y\otr x).\]
Then for any biquandle bracket $\beta$ defined by $A,B:X\times X\to R^{\times}$, 
the maps
\[A'_{x,y}=A_{x,y}\gamma(x,y)\quad \mathrm{and}\quad B'_{x,y}=B_{x,y}\gamma(x,y)\]
define a biquandle bracket $\gamma\beta$ with 
$\delta=-A_{x,y}B_{x,y}^{-1}-A_{x,y}^{-1}B_{x,y}$ and we have
$\Phi_X^{\beta}=\Phi_X^{\gamma\beta}$.
\end{proposition}

\begin{proof}
In $\gamma\beta$, the invertible quantity \[C(x)C(y\otr x)C(z)^{-1}C(y\utr z)^{-1}
C((x\utr y)\utr (z\otr y))^{-1}C((z\otr y)\otr (x\utr y))\]
factors out of each term on both sides of the equations in biquandle bracket 
axiom (iii), so $\gamma\beta$ is a biquandle bracket provided $\beta$ is. 

To see that $\beta$ and $\gamma\beta$ define the same invariant, note that
we can picture $\gamma\beta$ as including factors of $C(x),$ $C(y\otr x)$, 
$C(y)^{-1}$, $C(x\utr y)^{-1}$ and on the initial and terminal ends of the 
semiarc respectively along with the $A_{x,y}$ and $B_{x,y}$ coefficients as shown.
\[\includegraphics{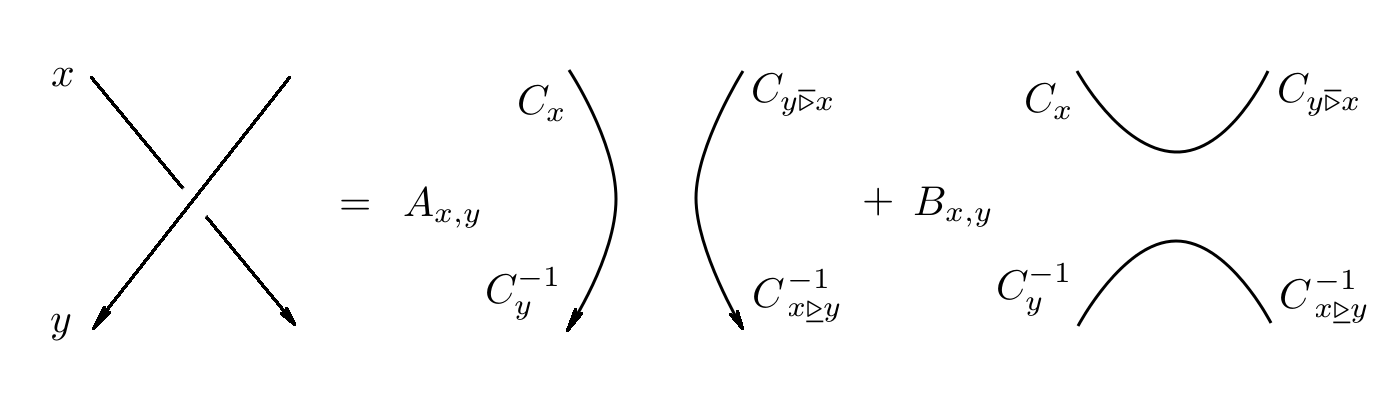}\]
\[\includegraphics{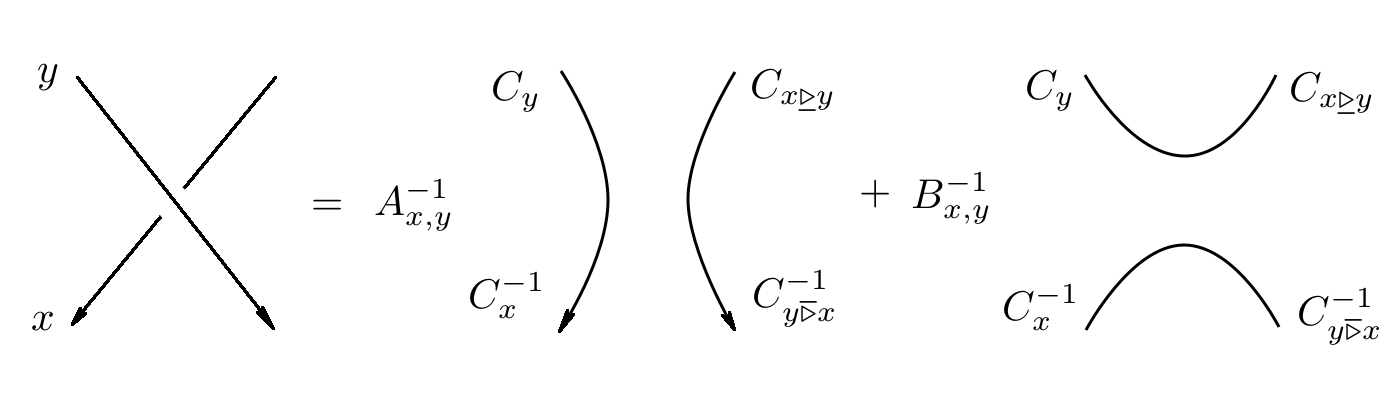}\]
Then we observe that over any complete link diagram, the $C$ factors match up 
in canceling pairs along each semiarc, so the value of each state of an 
$X$-colored link in $\Phi_X^{C\beta}$ is the same as in $\Phi_X^{\beta}$. 
\end{proof}

For biquandle brackets $\beta$ representing biquandle 2-cocycles, 
$\gamma$ is a coboundary and $\beta$ and $\gamma\beta$ are cohomologous; 
however, Proposition \ref{cohom} holds even for biquandle brackets $\beta$ not
representing cocycles. Thus, it is tempting to define $\beta$ and
$\gamma\beta$ to be ``cohomologous'' regardless of whether $\beta$ is a cocycle;
however, we will settle for the following:

\begin{definition}
Two $X$-brackets $\beta$ and $\beta'$ over $R$ are \textit{$C$-equivalent}
if there is a map $C:X\to R^{\times}$ such that for all $x,y\in X$, we have
\[\begin{array}{rcl}
A'_{x,y} & = & A_{x,y} C(x)C(y)^{-1}C(x\utr y)^{-1}C(y\otr x) \ \mathrm{and} \\
B'_{x,y} & = & B_{x,y} C(x)C(y)^{-1}C(x\utr y)^{-1}C(y\otr x).
\end{array}\]
\end{definition}

\begin{corollary}
$C$-equivalent $X$-brackets define the same invariant $\Phi_X^{\beta}$.
\end{corollary}

In \cite{NR}, \textit{quantum enhancements} of the counting invariant
with respect to involutory biquandles $X$ were defined as functors from
the category of $X$-labeled unoriented tangles to an $R$-module category.
Biquandle brackets provide examples of quantum enhancements as defined
in \cite{NR} in the following way: Given a biquandle bracket $\beta$,
define 
\[
I=\left[\begin{array}{rr} 1 & 0 \\ 0 & 1\end{array}\right],\quad
N=\left[\begin{array}{rrrr}
0 & A_{11} & -B_{11} & 0 
\end{array}\right]\quad\mathrm{and}
\quad
U=\left[\begin{array}{c}
0 \\ -B_{11}^{-1} \\ A_{11}^{-1} \\ 0 
\end{array}\right].
\]
Then the biquandle bracket skein relation yields $X$-labeled 
$R$-matrices $X_{x,y}^{\pm 1}$:
\begin{eqnarray*}
X_{x,y} & = &  A_{x,y}(I\otimes I)+B_{x,y}(UN)\\
& = & \left[\begin{array}{cccc}
A_{x,y} & 0 & 0 & 0 \\
0 & 0 & B_{x,y} & 0 \\
0 & B_{x,y} & A_{x,y}-A_{x,y}^{-1}B_{x,y}^2 & 0 \\
0 & 0 & 0 & A_{x,y}
\end{array}\right].
\end{eqnarray*}
See \cite{S} for more.

\begin{example}
The biquandle bracket in example \ref{ex1} corresponds to quantum weight
over $\mathbb{Z}_5$ given by
\[
I=\left[\begin{array}{rr} 1 & 0 \\ 0 & 1\end{array}\right],\quad
U=\left[\begin{array}{rrrr}
0 & 1 & 1 & 0 
\end{array}\right],\quad
N=\left[\begin{array}{c}
0 \\ 1 \\ 1 \\ 0 
\end{array}\right],
\]
\[X_{1,1}=\left[\begin{array}{cccc}
1 & 0 & 0 & 0 \\
0 & 0 & 4 & 0 \\
0 & 4 & 0 & 0 \\
0 & 0 & 0 & 1
\end{array}\right],\quad 
X_{1,2}=\left[\begin{array}{cccc}
3 & 0 & 0 & 0 \\
0 & 0 & 2 & 0 \\
0 & 2 & 1 & 0 \\
0 & 0 & 0 & 3
\end{array}\right],
\]
\[X_{2,1}=\left[\begin{array}{cccc}
4 & 0 & 0 & 0 \\
0 & 0 & 1 & 0 \\
0 & 1 & 0 & 0 \\
0 & 0 & 0 & 4
\end{array}\right],\quad\mathrm{and}\quad 
X_{2,2}=\left[\begin{array}{cccc}
1 & 0 & 0 & 0 \\
0 & 0 & 4 & 0 \\
0 & 4 & 0 & 0 \\
0 & 0 & 0 & 1
\end{array}\right].\]
In particular, this quantum enhancement is an example of a 
\textit{strongly heterogeneous} quantum weight as defined in the questions 
in \cite{NR}, since $X_{1,2}$ is not a classical $R$-matrix. 
\end{example}

\begin{example}
Let $X$ be the biquandle defined by the operation matrix
\[\left[\begin{array}{rr|rr}
2 & 2 & 2 & 2 \\
1 & 1 & 1 & 1
\end{array}\right]\] 
and let $R=\mathbb{F}_8=\mathbb{Z}_2[t]/(1+t+t^3)$ be the Galois field of
eight elements. That is, $R$ is the ring of polynomials in one variable
with $\mathbb{Z}_2$ coefficients with the rule that $t^3=1+t$. Then our
\texttt{python} computations reveal that
\[\left[\begin{array}{cc|cc}
1 & 1+t & t & t+t^2 \\
1+t^2 & 1 & 1 & t
\end{array}\right]\]
defines a biquandle bracket. We can describe this one
without explicitly referencing biquandles in the following way: Given any 
oriented link $L$ of $c$ components, find the the $2^c$ ways to color the 
semiarcs of $L$ alternately solid and dotted going around each component 
(or for virtual links lacking such a coloring, set the invariant value to 
zero). Then for each such coloring, expand using the following skein relations.\footnote{Thanks to Zhiqing Yang for catching a misprint in an earlier version of this table.}
\[\includegraphics{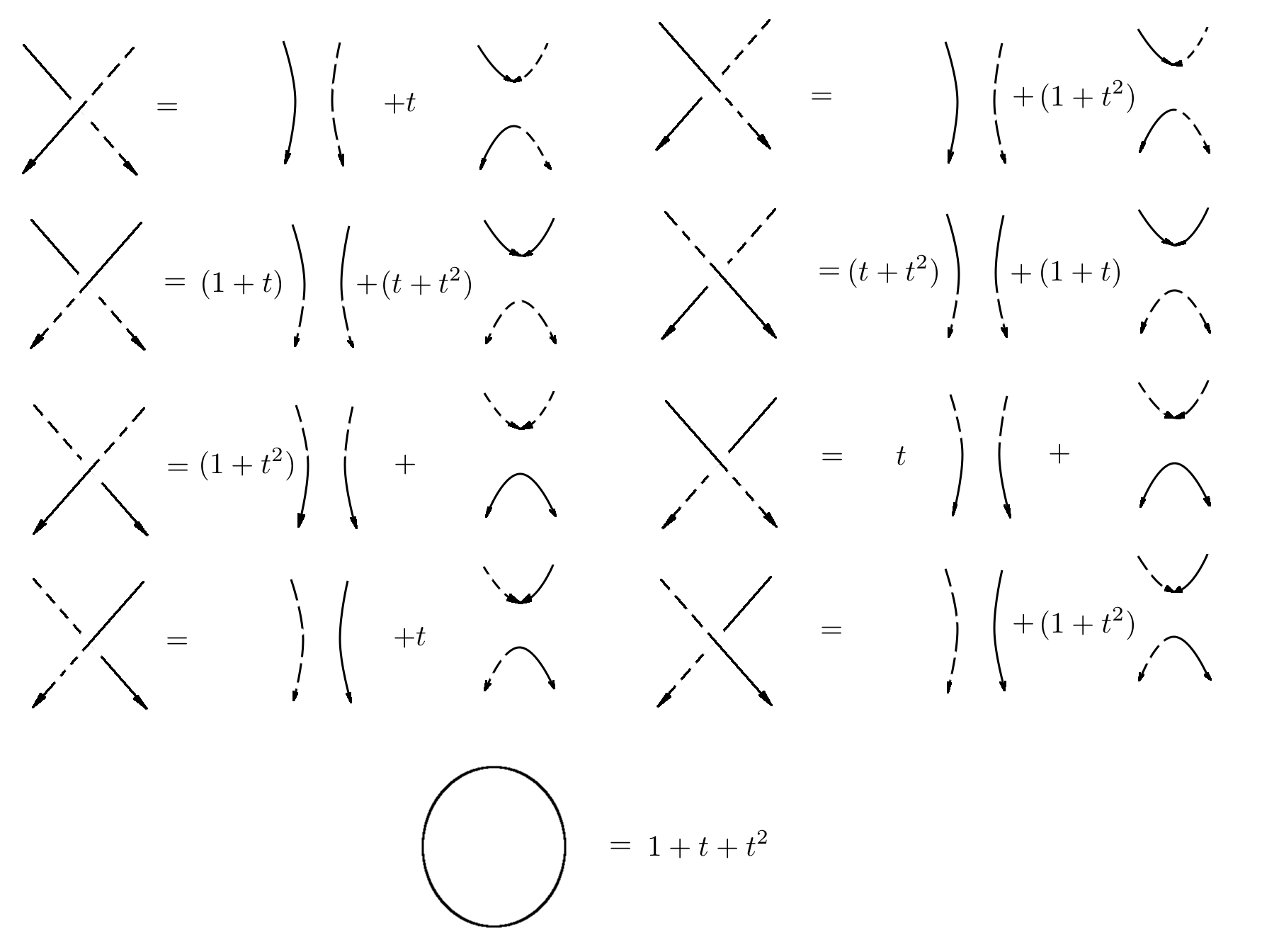}\]
Note that while fully resolved closed curves have some solid and some dashed 
sections and lack globally consistent orientations, every closed curve 
evaluates simply to $1+t+t^2$.
Finally, multiply by the writhe normalization factor $t^p(1+t^2)^n$
where $p$ and $n$ are the numbers of positive and negative crossings
respectively and collect these values into a multiset over the set of all
colorings of $L$.

We computed this invariant for all prime classical knots with up to eight 
crossings and all prime classical links with up to seven crossings
%, and all prime 
%virtual knots with up to four classical crossings 
as found in the tables at 
the knot atlas \cite{KA}. The results are collected in the tables below.
We list the multiset version of the invariant for ease of reading since this
ring is not $\mathbb{Z}$ or $\mathbb{Z}_n$.
We start with the prime classical knots:
\[\begin{array}{r|l}
\Phi_X^{\beta}(K) & K \\ \hline
\{2\times 1\} & 5_2, 7_5, 8_{10}, 8_{11}, 8_{13}, 8_{17} \\
\{2\times t\} & 3_1, 6_2, 8_9 \\
\{2\times 1+t\} & 4_1, 7_1, 7_4, 8_5, 8_{14} \\
\{2\times t^2\} & 6_1, 6_3, 7_2, 7_3, 8_7, 8_{21} \\
\{2\times 1+t^2\} & 7_7, 8_2, 8_3, 8_4, 8_8, 8_{19}, 8_{20} \\
\{2\times t+t^2\} & 8_1, 8_6, 8_{12}, 8_{16}, 8_{18} \\
\{2\times 1+t+t^2\} & \mathrm{Unknot}, 5_1, 7_6, 8_{15}
\end{array}\]
For prime classical links with up to seven crossings, we have
\[\begin{array}{r|l}
\Phi_X^{\beta}(L) & L \\ \hline
\{2\times 1, 2\times t\} & L6a2 \\
\{2\times t, 2\times 1+t^2\} & L7a6 \\
\{2\times 1, 2\times t+t^2\} & L6a1 \\
\{2\times 1, 2\times 1+t+t^2\} & L7a5 \\
\{2\times t^2, 2\times 1+t+t^2\} & L7a2, L7n1 \\
\{2\times 1+t^2, 2\times t+t^2\} & L2a1 \\
\{2\times 1+t^2, 2\times 1+t\} & L4a1 \\
\{2\times t+t^2, 2\times 1+t+t^2\} & L6a3 \\
\{4\times t^2\} & L5a1 \\
\{4\times t+t^2\} & L7a1, L7a3, L7a4 \\
\{2\times 1, 6\times t+t^2\} & L6a4, L6n1 \\
\{2\times t^2, 6\times 1+t^2\} & L6a5 \\
\{2\times t, 6\times t+t^2\} & L7a7. \\
\end{array}\]
%and for prime virtual knots with up to four classical crossings we have
%\[\begin{array}{r|l}
%\Phi_X^{\beta}(K) & K \\ \hline
%\{2\times 0\} & 2.1, 3.5, 4.2, 4.6, 4.8, 4.12, 4.17, 4.28, 4.32, 4.51, 4.58, 4.%71, 4.75, 4.89, 4.105 \\
%\{2\times 1\} & 4.23, 4.41, 4.65, 4.79 \\
%\{2\times t\} & 3.6, 4.15, 4.16, 4.20, 4.22, 4.34, 4.40, 4.52, 4.60, 4.64, 4.82, 4.87, 4.92, 4.94\\
%\{2\times 1+t\} & 4.9, 4.10, 4.29, 4.31, 4.37, 4.48, 4.50, 4.57, 4.61, 4.69, 4.70, 4.78, 4.86, 4.90, 4.99, 4.108 \\
%\{2\times t^2\} & 3.3, 4.4, 4.5, 4.11, 4.18, 4.25, 4.30, 4.33, 4.38, 4.39, 4.43,, 4,44, 4.45, 4.49, 4.54, 4.62, \\
% & 4.63, 4.74, 4.80, 4.83, 4.84, 4.88, 4.91, 4.95, 4.100, 4.101, 4.104 \\
%\{2\times 1+t^2\} & 4.1, 4.3, 4.7, 4.21, 4.24, 4.36, 4.53, 4.68, 4.73 \\
%\{2\times t+t^2\} & 3.2, 3.4, 4.27, 4.81 \\
%\{2\times 1+t+t^2\} & 3.1, 3.7, 4.13, 4.19, 4.26, 4.35, 4.42, 4.46, 4.47, 4.55,% 4.56, 4.59, 4.66, 4.67, 4.72, 4.76, \\
% & 4.77, 4.85, 4.93, 4.96, 4.97, 4.98, 4.102, 4.103, 4.106, 4.107
%\end{array}\]
We note that:
\begin{itemize}
\item $\Phi_X^{\beta}$ distinguishes the right- and left-hand trefoils, with
invariant values of $\{2\times t\}$ and $\{2\times 0\}$ respectively and
hence can distinguish mirror images,
\item $\Phi_X^{\beta}$ distinguishes the 
Square knot from the Granny knot with invariant values of $\{2\times t+t^2\}$
and $\{2\times 0\}$ respectively, so $\Phi_X^{\beta}$ is not determined
by the knot group, and
%\item $\Phi_X^{\beta}(10_{132})=\{2\times t+t^2\}\ne 
%\{2\times 1+t+t^2\}=\Phi_X^{\beta}(5_1)$
%and hence $\Phi_X^{\beta}$ is not determined by the HOMFLY-PT, Jones, 
%or Alexander polynomials. 
\item If we compute this invariant via the skein expansion rather than the 
state-sum method, it is important to freeze the diagram in place and not 
change any smoothed diagrams by Reidemeister moves, since these can change 
the value of the invariant.
\end{itemize}
\end{example}

\section{\large\textbf{Quandle Brackets}}\label{Qb}

Let $X$ be a \textit{quandle}, that is, a biquandle with $x\otr y=x$ for all
$x,y\in X$. An $X$-bracket in this case is called a \textit{quandle bracket}.

\begin{example}
Consider the dihedral quandle $X$ on three elements, with operation table
\[\left[\begin{array}{rrr|rrr}
1 & 3 & 2 & 1 & 1 & 1 \\
3 & 2 & 1 & 2 & 2 & 2 \\
2 & 1 & 3 & 3 & 3 & 3
\end{array}\right].\]
Then $X$-bracket over $\mathbb{Z}_{11}$ include $\beta$ given by
\[\left[\begin{array}{rrr|rrr}
1 & 7 & 7 & 7 & 5 & 5 \\
1 & 1 & 8 & 7 & 7 & 1 \\
1 & 8 & 1 & 7 & 1 & 7
\end{array}\right].\]
This is not a quandle $2$-cocycle since $A_{x,y}\ne B_{x,y}$; our
\texttt{python} code computed the following values for prime knots with up to
eight crossings:% and links with up to seven crossings.
\[\begin{array}{r|l}
\Phi_X^{\beta}(K) & K \\ \hline
3 & 6_2 \\
3u & 5_2, 7_3, 7_6, 8_1,8_2 \\
3u^2 & 4_1, \\
3u^3 & 5_1, 8_8 \\
3u^5 & 7_1, 7_5, 8_6,8_{12} \\
3u^6 & 6_3, 8_3, 8_{14} \\
3u^7 & \mathrm{Unknot}, 3u^7 \\ 
3u^9 & 7_2, 8_7, 8_{13}, 8_{17} \\
3u^{10} & 8_9, 8_{16} \\
9u & 8_{20} \\
9u^2 & 8_{10} \\
9u^3 & 8_{11}, 8_{15} \\
9u^5 & 7_4,8_5 \\
9u^7 & 8_{19} \\
9u^8 & 7_7 \\
9u^{10} & 8_{21} \\
27u^2 & 8_{18}.
\end{array}\]
\end{example}

\begin{proposition}\label{qcond}
If $X$ is a quandle and $R$ is a commutative ring, then maps 
$A,B:X\times X\to R^{\times}$ defining a quandle bracket must
satisfy the \textit{mixed cocycle conditions}
\[\begin{array}{rcll}
A_{x,y}A_{x\utr y,z} & = & A_{x,z}A_{x\utr z,y\utr z} & (i) \\
A_{x,y}B_{x\utr y,z} & = & B_{x,z}A_{x\utr z,y\utr z} & (ii)\\
B_{x,y}A_{x\utr y,z} & = & A_{x,z}B_{x\utr z,y\utr z} & (iii)\\
B_{x,y}B_{x\utr y,z} & = & B_{x,z}B_{x\utr z,y\utr z} & (iv).
\end{array}\]
\end{proposition}

\begin{proof}
Suppose our biquandle $X$ is a quandle, i.e., $x\otr y=x$ for all $x,y\in X$.
Then the first three biquandle bracket conditions from the Reidemeister III 
move reduce to
\[\begin{array}{rcl}
A_{x,y}A_{y,z}A_{x\utr y,z} & = & A_{x,z}A_{y,z}A_{x\utr z,y\utr z} \\
A_{x,y}B_{y,z}B_{x\utr y,z} & = & B_{x,z}B_{y,z}A_{x\utr z,y\utr z} \\
B_{x,y}A_{y,z}B_{x\utr y,z} & = & B_{x,z}A_{y,z}B_{x\utr z,y\utr z} \\
\end{array}\quad \Rightarrow \quad
\begin{array}{rcl}
A_{x,y}A_{x\utr y,z} & = & A_{x,z}A_{x\utr z,y\utr z} \\
A_{x,y}B_{x\utr y,z} & = & B_{x,z}A_{x\utr z,y\utr z} \\
B_{x,y}B_{x\utr y,z} & = & B_{x,z}B_{x\utr z,y\utr z} \\
\end{array}
\]
yielding (i),(ii) and (iv). Then the remaining biquandle bracket equations say
\begin{eqnarray*}
A_{y,z}(A_{x,y}B_{x\utr y,z}-A_{x,z}B_{x\utr z,y\utr z}) & = & 
B_{y,z}(A_{x,z}A_{x\utr z,y\utr z}+\delta A_{x,z}B_{x\utr z,y\utr z}+B_{x,z}B_{x\utr z,y\utr z}) \\
A_{y,z}(B_{x,y}A_{x\utr y,z}-B_{x,z}A_{x\utr z,y\utr z}) & = & 
-B_{y,z}(A_{x,y}A_{x\utr y,z}+\delta B_{x,y}A_{x\utr y,z}+B_{x,y}B_{x\utr y,z})
\end{eqnarray*}
which then implies
\[
A_{y,z}(B_{x,y}A_{x\utr y,z}-A_{x,z}B_{x\utr z,y\utr z})  
=  \delta B_{y,z}(A_{x,z}B_{x\utr z,y\utr z} -A_{x,y}B_{x\utr y,z})\]
so we have
\[(B_{x,y}A_{x\utr y,z}-A_{x,z}B_{x\utr z,y\utr z})(A_{x_,y}+\delta B_{y,z})=0.\]
Then 
\[A_{y,z}+\delta B_{y,z} 
=A_{y,z}+(-A_{yz}B_{y,z}^{-1}-A_{y,z}^{-1}B_{y,z})B_{y,z}
=-A_{y,z}^{-1}B_{y,z}^2\]
is a unit in $R$, so $B_{x,y}A_{x\utr y,z}-A_{x,z}B_{x\utr z,y\utr z}=0$ as required. 
\end{proof}

We note that the converse to proposition \ref{qcond} is not true -- the
mixed cocycle conditions are necessary but not sufficient conditions for
maps $A:X\times X\to R$ to define a quandle bracket, as the next example 
demonstrates.

\begin{example}
Consider the trivial quandle on two elements, $T_2=\{1,2\}$ with 
$x\utr y=x\otr y=x$. The maps $A,B:X\times X\to \mathbb{Z}_3$ defined by
\[\left[\begin{array}{rr|rr}
1 & 1 & 1 & 1\\
1 & 2 & 1 & 2
\end{array}\right]\]
satisfy all four mixed cocycle conditions and also the conditions that
\[\delta=-A_{x,y}B_{x,y}^{-1}-A_{x,y}^{-1}B_{x,y}=-2=1\]
and
\[w=2=A_{x,x}\delta+B_{x,x}\]
for all $x,y\in X$; however, this is not a biquandle bracket since
$A_{1,2}A_{2,2}B_{12}=2$ but 
\[A_{1,2}B_{2,2}A_{1,2}+A_{1,2}A_{2,2}B_{1,2}-2A_{1,2}B_{2,2}B_{12}+B_{1,2}B_{2,2}B_{1,2}=4=1\ne 2\]
so the fourth equation in biquandle bracket axiom (iii) is not satisfied.
\end{example}

\section{\large\textbf{Questions}}\label{Q}

We end with some questions for future research. This is second paper in an
ongoing series on quantum enhancements; future papers are underway extending
the present results to knotted surface and virtual knots in various ways.

What exactly is the relationship between biquandle and brackets biquandle 
cohomology? Is there a generalized theory of biquandle cohomology which 
includes those biquandle brackets which are not biquandle cocycles in
the traditional sense? Are there quantum enhancements which do not arise 
from biquandle brackets? What Khovanov homology-style categorifications
of biquandle bracket invariants are possible? What about 
biquandle-colored skein modules?

\bibliography{sn-mo-vr-rev3}{}
\bibliographystyle{abbrv}

%\begin{thebibliography}{10}
%
%\bibitem{KA} D. Bar-Natan. The Knot Atlas.
%\texttt{http://katlas.math.toronto.edu/wiki/Main\underline{\ }Page}
%
%\bibitem{FJK}{R. Fenn, M. Jordan-Santana and L. Kauffman. Biquandles
%and virtual links.  \textit{Topology Appl.}  \textbf{145}  (2004) 157--175.}
%
%\bibitem{FRS}{R. Fenn, B. Sanderson and C. Rourke. An introduction to Species 
%and the Rack Space. Topics in Knot Theory: Kluwer Academic Publishers (1993) 
%33--55.}
%
%\bibitem{K3}{L. Kauffman. Virtual Knot Theory. \textit{European J. Combin.}
%\textbf{20} (1999) 663-690.}
%
%\bibitem{KR}{L. H. Kauffman and D. Radford. Bi-oriented quantum algebras, 
%and a generalized Alexander polynomial for virtual links. 
%\textit{Contemp. Math}. \textbf{318} (2003) 113-140.}
%
%\bibitem{NR}{S. Nelson and V. Rivera.
%Quantum enhancements of involutory birack counting invariants. 
%\textit{J. Knot Theory Ramifications} \textbf{23} (2014) Article ID 1460006, 15 p.}
%
%\end{thebibliography}

\bigskip

\noindent
\textsc{Department of Mathematical Sciences \\
Claremont McKenna College \\
850 Columbia Ave. \\
Claremont, CA 91711} 

\bigskip

\noindent
\textsc{Department of Mathematics \\
Harvey Mudd College\\
301 Platt Boulevard \\
Claremont, CA 91711
}

\end{document}